\documentclass[oneside,11pt,reqno]{conm-p-l}

\usepackage{verbatim,upref,amsxtra,amssymb,amscd,graphicx}

\usepackage{mathtools}

\usepackage{color}

\usepackage[colorlinks=true, linkcolor=blue, citecolor=blue]{hyperref}

\DeclareMathAlphabet\mathscr{U}{eus}{m}{n}
\SetMathAlphabet\mathscr{bold}{U}{eus}{b}{n}
\DeclareMathAlphabet\matheur{U}{eur}{m}{n}
\SetMathAlphabet\matheur{bold}{U}{eur}{b}{n}

\numberwithin{equation}{section}

\newtheorem{theo}{Theorem}[section]
\newtheorem{prop}[theo]{Proposition}
\newtheorem{lemm}[theo]{Lemma}
\newtheorem{coro}[theo]{Corollary}

\theoremstyle{definition}

\newtheorem{defi}[theo]{Definition}

\theoremstyle{remark}

\newtheorem{rema}[theo]{Remark}

\newcommand{\rc}{\mathrm{center}}

\begin{document}\allowdisplaybreaks\frenchspacing

\setlength{\baselineskip}{1.1\baselineskip}

\title{Ergodicity of principal algebraic group actions}

\author{Hanfeng Li}

\author{Jesse Peterson}

\author{Klaus Schmidt}

\address{Hanfeng Li: Department of Mathematics, Chongqing University,
Chong\-qing 401331, China \newline\indent \textup{and} \newline\indent
Department of Mathematics, SUNY at Buffalo,
Buffalo, NY 14260-2900, U.S.A.}
\email{hfli@math.buffalo.edu}

\address{Jesse Peterson: Department of Mathematics, Vanderbilt University, 1326 Stevenson Center, Nashville, TN 37240, U.S.A.} \email{jesse.d.peterson@vanderbilt.edu}

\address{Klaus Schmidt: Mathematics Institute, University of Vienna, Nordberg\-stra{\ss}e 15, A-1090 Vienna, Austria \newline\indent \textup{and} \newline\indent Erwin Schr\"odinger Institute for Mathematical Physics, Boltzmanngasse~9, A-1090 Vienna, Austria} \email{klaus.schmidt@univie.ac.at}

\thanks{H.L. was partially supported by the NSF grants DMS-1001625 and DMS-126623, and he would like to thank the Erwin Schr\"{o}dinger Institute, Vienna, for hospitality and support while some of this work was done,\\\indent J.P. was partially supported by the NSF grant DMS-1201565 and the Alfred P. Sloan Foundation,\\\indent Both J.P. and K.S. would like to thank the University of Buffalo for hospitality and support while some of this work was done.}
\subjclass[2000]{}

\dedicatory{Dedicated to Shrikrishna Gopalrao Dani on the occasion of his 65th birthday}


\begin{abstract}An \textit{algebraic} action of a discrete group $\Gamma $ is a homomorphism from $\Gamma $ to the group of continuous automorphisms of a compact abelian group $X$. By duality, such an action of $\Gamma $ is determined by a module $M=\widehat{X}$ over the integer group ring $\mathbb{Z}\Gamma $ of $\Gamma $. The simplest examples of such modules are of the form $M=\mathbb{Z}\Gamma /\mathbb{Z}\Gamma f$ with $f\in \mathbb{Z}\Gamma $; the corresponding algebraic action is the \textit{principal algebraic $\Gamma $-action} $\alpha _f$ defined by $f$.

In this note we prove the following extensions of results by Hayes \cite{Hayes} on ergodicity of principal algebraic actions: If $\Gamma $ is a countably infinite discrete group which is not virtually cyclic, and if $f\in\mathbb{Z}\Gamma $ satisfies that right multiplication by $f$ on $\ell ^2(\Gamma ,\mathbb{R})$ is injective, then the principal $\Gamma $-action $\alpha _f$ is ergodic (Theorem \ref{t:ergodic2}). If $\Gamma $ contains a finitely generated subgroup with a single end (e.g. a finitely generated amenable subgroup which is not virtually cyclic), or an infinite nonamenable subgroup with vanishing first $\ell ^2$-Betti number (e.g., an infinite property $T$ subgroup), the injectivity condition on $f$ can be replaced by the weaker hypothesis that $f$ is not a right zero-divisor in $\mathbb{Z}\Gamma $ (Theorem \ref{t:ergodic1}). Finally, if $\Gamma $ is torsion-free, not virtually cyclic, and satisfies Linnell's \textit{analytic zero-divisor conjecture}, then $\alpha _f$ is ergodic for every $f\in \mathbb{Z}\Gamma $ (Remark \ref{r:analytic zero divisor}).\end{abstract}

\maketitle

\section{Principal Algebraic Group Actions}\label{s:introduction}

Let $\Gamma $ be a countably infinite discrete group with integral group ring $\mathbb{Z}\Gamma $. Every $g\in\mathbb{Z}\Gamma $ is written as a formal sum $g =\sum_{\gamma }g_{\gamma }\cdot \gamma $, where $g_{\gamma }\in\mathbb{Z}$ for every $\gamma \in\Gamma $ and $\sum_{\gamma \in\Gamma }|g_{\gamma }|<\infty $. The set $\textup{supp}(g)=\{\gamma \in \Gamma :g_\gamma \ne 0\}$ is called the \textit{support} of $g$. For $g=\sum_{\gamma \in\Gamma }g_\gamma \cdot \gamma \in\mathbb{Z}\Gamma $ we denote by $g^*=\sum_{\gamma \in\Gamma }g_\gamma \cdot \gamma ^{-1}$ the \textit{adjoint} of $g$. The map $g\mapsto g^*$ is an \textit{involution} on $\mathbb{Z}\Gamma $, i.e., $(gh)^*=h^*g^*$ for all $g,h\in\mathbb{Z}\Gamma $, where the product $fg$ of two elements $f=\sum_{\gamma }f_\gamma \cdot \gamma $ and $g=\sum_\gamma g_\gamma \cdot \gamma $ in $\mathbb{Z}\Gamma $ is given by $fg=\sum_{\gamma ,\gamma '\in\Gamma }f_\gamma g_{\gamma '}\cdot \gamma \gamma '$.

An \textit{algebraic $\Gamma $-action} is a homomorphism $\alpha \colon \Gamma \longrightarrow \textup{Aut}(X)$ from $\Gamma $ to the group of (continuous) automorphisms of a compact metrizable abelian group $X$. If $\alpha $ is an algebraic $\Gamma $-action, then $\alpha ^\gamma \in \textup{Aut}(X)$ denotes the image of $\gamma \in \Gamma $, and $\alpha ^{\gamma \gamma '}=\alpha ^\gamma \alpha ^{\gamma '}$ for every $\gamma ,\gamma '\in \Gamma $. The $\Gamma $-action $\alpha $ induces an action of $\mathbb{Z}\Gamma $ by group homomorphisms $\alpha ^f\colon X\longrightarrow X$, where $\alpha ^f=\sum_{\gamma \in\Gamma }f_\gamma \alpha ^\gamma $ for every $f=\sum_{\gamma \in\Gamma }f_\gamma \cdot \gamma \in\mathbb{Z}\Gamma $. Clearly, if $f,g\in\mathbb{Z}\Gamma $, then $\alpha ^{fg}=\alpha ^f \alpha ^g$.

Let $\hat{X}$ be the dual group of $X$. If $\hat{\alpha }^\gamma $ is the automorphism of $\hat{X}$ dual to $\alpha ^\gamma $, then the map $\hat{\alpha }\colon \Gamma \longrightarrow \textup{Aut}(\hat{X})$ satisfies that $\hat{\alpha }^{\gamma \gamma '}=\hat{\alpha }^{\gamma '} \hat{\alpha }^\gamma $ for all $\gamma ,\gamma '\in \Gamma $. We write $\hat{\alpha }^f\colon \hat{X}\longrightarrow \hat{X}$ for the group homomorphism dual to $\alpha ^f$ and set $f\cdot a=\hat{\alpha }^{f^*}a$ for every $f\in\mathbb{Z}\Gamma $ and $a\in \hat{X}$. The resulting map $(f,a)\mapsto f\cdot a$ from $\mathbb{Z}\Gamma \times \hat{X}$ to $\hat{X}$ satisfies that $(fg)\cdot a=f\cdot (g\cdot a)$ for all $f,g\in\mathbb{Z}\Gamma $ and turns $\hat{X}$ into a module over the group ring $\mathbb{Z}\Gamma $. Conversely, if $M$ is a countable module over $\mathbb{Z}\Gamma $, we set $X=\widehat{M}$ and put $\hat{\alpha }^fa=f^*\cdot a$ for $f\in\mathbb{Z}\Gamma $ and $a\in M$. The maps $\alpha ^f\colon \widehat{M}\longrightarrow \widehat{M}$ dual to $\hat{\alpha }^f,\,f\in\mathbb{Z}\Gamma $, define an action of $\mathbb{Z}\Gamma $ by homomorphisms of $\widehat{M}$, which in turn induces an algebraic action $\alpha $ of $\Gamma $ on $X=\widehat{M}$.

The simplest examples of algebraic $\Gamma $-actions arise from $\mathbb{Z}\Gamma $-modules of the form $M=\mathbb{Z}\Gamma /\mathbb{Z}\Gamma f$ with $f\in\mathbb{Z}\Gamma $. Since these actions are determined by principal left ideals of $\mathbb{Z}\Gamma $ they are called \textit{principal algebraic $\Gamma $-actions}. In order to describe these actions more explicitly we put $\mathbb{T}=\mathbb{R}/\mathbb{Z}$ and define the left and right shift-actions $\lambda $ and $\rho $ of $\Gamma $ on $\mathbb{T}^\Gamma $ by setting
	\begin{equation}
	\label{eq:lambda}
(\lambda ^\gamma x)_{\gamma '}=x_{\gamma ^{-1}\gamma '},\qquad (\rho ^\gamma x)_{\gamma '}=x_{\gamma '\gamma },
	\end{equation}
for every $\gamma \in \Gamma $ and $x=(x_{\gamma '})_{\gamma '\in \Gamma }\in\mathbb{T}^\Gamma $. The $\Gamma $-actions $\lambda $ and $\rho $ extend to actions of $\mathbb{Z}\Gamma $ on $\mathbb{T}^\Gamma $ given by
	\begin{equation}
	\label{eq:Lambda}
\lambda ^f=\textstyle\sum_{\gamma \in\Gamma }f_\gamma \lambda ^\gamma ,\qquad \rho ^f=\textstyle\sum_{\gamma \in\Gamma }f_\gamma \rho ^\gamma
	\end{equation}
for every $f= \sum_{\gamma \in\Gamma }f_\gamma \cdot \gamma \in\mathbb{Z}\Gamma $.

\smallskip The pairing $\langle f,x\rangle =e^{2\pi i\sum_{\gamma \in \Gamma }f_\gamma x_\gamma }$, $f=\sum_{\gamma \in \Gamma }f_\gamma \cdot \gamma \in\mathbb{Z}\Gamma $, $x=(x_\gamma )\in\mathbb{T}^\Gamma $, identifies $\mathbb{Z}\Gamma $ with the dual group $\widehat{\mathbb{T}^\Gamma }$ of $\mathbb{T}^\Gamma $. We claim that, under this identification,
	\begin{equation}
	\label{eq:Xf}
X_f\coloneqq\ker \rho ^f = \bigl\{x\in \mathbb{T}^\Gamma :\rho ^fx=\textstyle\sum_{\gamma \in \Gamma }f_\gamma \rho ^\gamma x=0\bigr\}=(\mathbb{Z}\Gamma f)^\perp\subset \widehat{\mathbb{Z}\Gamma }=\mathbb{T}^\Gamma .
	\end{equation}
Indeed,
	\begin{align*}
\langle h,\rho ^fx\rangle &=\smash[t]{\Bigl\langle h,\sum\nolimits_{\gamma '\in \Gamma }f_{\gamma '}\rho ^{\gamma '}x\Bigr\rangle =e^{2\pi i\sum_{\gamma \in \Gamma }h_\gamma \sum_{\gamma '\in \Gamma }f_{\gamma '}x_{\gamma \gamma '}}}
	\\
&=e^{2\pi i\sum_{\gamma \in \Gamma }\sum\nolimits_{\gamma '\in \Gamma }h_{\gamma \gamma '^{-1}}f_{\gamma '}x_\gamma} =e^{2\pi i\sum_{\gamma \in \Gamma }(hf)_\gamma x_\gamma} =\langle hf,x\rangle
	\end{align*}
for every $h\in\mathbb{Z}\Gamma $ and $x\in\mathbb{T}^\Gamma $, so that $x\in\ker \rho ^f$ if and only if $x\in (\mathbb{Z}\Gamma f)^\perp$.

Since the $\Gamma $-actions $\lambda $ and $\rho $ on $\mathbb{T}^\Gamma $ commute, the group $X_f=\ker\rho ^f\subset \mathbb{T}^\Gamma $ is invariant under $\lambda $, and we denote by $\alpha _f$ the restriction of $\lambda $ to $X_f$. In view of this we adopt the following terminology.

	\begin{defi}
	\label{d:principal}
$(X_f,\alpha _f)$ is the principal algebraic $\Gamma $-action defined by $f\in\mathbb{Z}\Gamma $.
	\end{defi}

In \cite{Hayes} the author calls a countably infinite discrete group $\Gamma $ \textit{principally ergodic} if every principal algebraic $\Gamma $-action $\alpha _f,\;f\in \mathbb{Z}\Gamma $, is ergodic w.r.t. Haar measure on $X_f$ and proves that the following classes of groups are principally ergodic: torsion-free nilpotent groups which are not virtually cyclic,\footnote{A discrete group $\Gamma $ is \textit{virtually cyclic} if it has a cyclic finite-index subgroup. Virtually cyclic groups can obviously not be principally ergodic: if $\Gamma =\mathbb{Z}$, and if $\mathbb{Z}\Gamma $ is identified with the ring of Laurent polynomials $\mathbb{Z}[u^{\pm1}]$ in the obvious manner, then the principal algebraic $\mathbb{Z}$-action $\alpha _f$ defined by $f=1-u$ is trivial --- and hence nonergodic --- on $X_f=\mathbb{T}$.} free groups on more than one generator, and groups which are not finitely generated.

In order to state our extensions of these results we denote by $\ell ^\infty (\Gamma ,\mathbb{R})\subset \mathbb{R}^\Gamma $ the space of bounded real-valued maps $v=(v_\gamma )$ on $\Gamma $, where $v_\gamma $ is the value of $v$ at $\gamma $, and we write $\|v\|_\infty =\sup_{\gamma \in\Gamma }|v_\gamma |$ for the supremum norm on $\ell ^\infty (\Gamma ,\mathbb{R})$. For $1\le p <\infty $ we set $\ell ^p(\Gamma ,\mathbb{R})=\{v=(v_\gamma )\in\ell ^\infty (\Gamma ,\mathbb{R}):\|v\|_p=\bigl(\sum_{\gamma \in\Gamma }|v_\gamma |^p\bigr)^{1/p}<\infty \}$. By $\ell ^p(\Gamma ,\mathbb{Z})=\ell ^p(\Gamma ,\mathbb{R})\cap \mathbb{Z}^\Gamma $ we denote the additive subgroup of integer-valued elements of $\ell ^p(\Gamma ,\mathbb{R})$; for $1\le p<\infty $, $\ell ^p(\Gamma ,\mathbb{Z})=\ell ^1(\Gamma ,\mathbb{Z})$ is identified with $\mathbb{Z}\Gamma $ by viewing each $g=\sum_{\gamma }g_\gamma \cdot \gamma \in\mathbb{Z}\Gamma $ as the element $(g_\gamma )_{\gamma \in\Gamma }\in \ell ^1(\Gamma ,\mathbb{Z})$.

The group $\Gamma $ acts on $\ell ^p(\Gamma ,\mathbb{R})$ isometrically by left and right translations: for every $v\in\ell ^p(\Gamma ,\mathbb{R})$ and $\gamma \in\Gamma $ we denote by $\tilde{\lambda }^\gamma v$ and $\tilde{\rho }^\gamma v$ the elements of $\ell ^p(\Gamma ,\mathbb{R})$ satisfying $(\tilde{\lambda }^\gamma v)_{\gamma '}=v_{\gamma ^{-1}\gamma '}$ and $(\tilde{\rho }^\gamma v)_{\gamma '}=v_{\gamma '\gamma }$, respectively, for every $\gamma '\in\Gamma $. Note that $\tilde{\lambda }^{\gamma \gamma '}=\tilde{\lambda }^\gamma \tilde{\lambda }^{\gamma '}$ and $\tilde{\rho }^{\gamma \gamma '}=\tilde{\rho }^\gamma \tilde{\rho }^{\gamma '}$ for every $\gamma ,\gamma '\in\Gamma $.

The $\Gamma $-actions $\tilde{\lambda }$ and $\tilde{\rho }$ extend to actions of $\ell ^1(\Gamma ,\mathbb{R})$ on $\ell ^p(\Gamma ,\mathbb{R})$ which will again be denoted by $\tilde{\lambda }$ and $\tilde{\rho }$: for $h=(h_\gamma )\in \ell ^1(\Gamma ,\mathbb{R})$ and $v\in\ell ^p(\Gamma ,\mathbb{R})$ we set
	\begin{equation}
	\label{eq:representation1}
\tilde{\lambda }^hv=\sum\nolimits_{\gamma \in \Gamma }h_\gamma \tilde{\lambda }^\gamma v,\qquad \tilde{\rho }^hv=\sum\nolimits_{\gamma \in \Gamma }h_\gamma \tilde{\rho }^\gamma v.
	\end{equation}
These definitions correspond to the usual convolutions
	\begin{equation}
	\label{eq:representation2}
\tilde{\lambda }^hv=h\cdot v,\qquad \tilde{\rho }^hv=v\cdot h^*,
	\end{equation}
where $h\mapsto h^*$ is the involution on $\ell ^1(\Gamma ,\mathbb{\mathbb{C}})$ defined as for $\mathbb{Z}\Gamma $: $h^*_\gamma =\overline{h_{\gamma ^{-1}}},\;\gamma \in \Gamma $, for every $h=(h_\gamma )\in \ell ^1(\Gamma ,\mathbb{C})$. For $p=2$, the bounded linear operators $\tilde{\lambda }^h,\tilde{\rho }^h\colon \ell ^2(\Gamma ,\mathbb{R})\longrightarrow \ell ^2(\Gamma ,\mathbb{R})$ in \eqref{eq:representation1} can be viewed as elements of the right (resp. left) equivariant group von Neumann algebra of $\Gamma $.

	\begin{theo}
	\label{t:ergodic1}
Let $\Gamma $ be a countably infinite discrete group which satisfies one of the following conditions:
	\begin{enumerate}
	\item
$\Gamma $ contains a finitely generated amenable subgroup which is not virtually cyclic, or more generally, a finitely generated subgroup with a single end,
	\item
$\Gamma $ is not finitely generated,
	\item
$\Gamma $ contains an infinite property T subgroup, or more generally, a nonamenable subgroup $\Gamma_0$ with vanishing first $\ell^2$-Betti number \linebreak $\beta_1^{(2)}(\Gamma_0) = 0$.
	\end{enumerate}
If $f\in\mathbb{Z}\Gamma $ is not a right zero-divisor, then the principal $\Gamma $-action $\alpha _f$ on $X_f$ is ergodic (with respect to the normalized Haar measure of $X_f$).
	\end{theo}

	\begin{theo}
	\label{t:ergodic2}
Let $\Gamma $ be a countably infinite discrete group which is not virtually cyclic. If $f\in\mathbb{Z}\Gamma $ satisfies that
	\begin{equation}
	\label{eq:kernel}
\ker\tilde{\rho }^{f^*}=\{v\in\ell ^2(\Gamma ,\mathbb{R}):\tilde{\rho }^{f^*}(v) = v\cdot f=0\}=\{0\},
	\end{equation}
then the principal $\Gamma $-action $\alpha _f$ on $X_f$ is ergodic.
	\end{theo}

In view of the hypotheses on $f$ in the Theorems \ref{t:ergodic1} and \ref{t:ergodic2} it is useful to recall the following result.

	\begin{prop}
	\label{p:divisor}
Let $\Gamma $ be a countably infinite discrete amenable group. For every $f\in\mathbb{Z}\Gamma $ the following conditions are equivalent.
	\begin{enumerate}
	\item
$f$ is a right zero-divisor in $\mathbb{Z}\Gamma $,
    \item
$\{v\in \ell^2(\Gamma, \mathbb{R}): f^*\cdot v=0\}\neq \{0\}$,
	\item
$f$ is a left zero-divisor in $\mathbb{Z}\Gamma $,
    \item
$\{v\in \ell^2(\Gamma, \mathbb{R}): f\cdot v=0\}\neq \{0\}$,
	\item
$\ker\tilde{\rho }^f=\{v\in\ell ^2(\Gamma ,\mathbb{R}):\tilde{\rho }^f(v)=0\}\ne\{0\}$.
	\end{enumerate}
	\end{prop}

{\begin{proof} $(4)\Leftrightarrow(5)$: This follows from $(f\cdot v)^*=v^*\cdot f^*$ for all $v\in \ell^2(\Gamma, \mathbb{R})$.

$(2)\Leftrightarrow(3)\Rightarrow(4)$: This is part of \cite[Proposition 4.16]{LT}.

$(1)\Leftrightarrow(2)$: Taking $*$ we see that (1) holds if and only if $f^*$ is a left zero-divisor in $\mathbb{Z}\Gamma$. Applying $(3)\Leftrightarrow(4)$ to $f^*$, we see that the latter condition is equivalent to (2).
	\end{proof}}

	\begin{rema}
\label{r:analytic zero divisor}
Linnell's \textit{analytic zero-divisor conjecture} is the conjectural statement that for any torsion-free discrete group $\Gamma $ and any nonzero $f\in \mathbb{C}\Gamma $, $\ker \tilde{\rho }^{f^*}=\{0\}$ \cite[Conjecture 1]{Linnell92}. Linnell has shown that this conjecture holds for $\Gamma $ if $G_1$ is a normal subgroup of $\Gamma $, $G_2$ is a normal subgroup of $G_1$, $\Gamma $ is torsion-free, $G_2$ is free, $G_1/G_2$ is elementary amenable, and $\Gamma /G_1$ is right orderable \cite[Proposition 1.4]{Linnell93}.

\textit{If a countably infinite, torsion-free, and not virtually cyclic group $\Gamma $ satisfies Linnell's analytic zero-divisor conjecture, then the principal $\Gamma $-action $\alpha _f$ on $X_f$ is ergodic for every $f\in\mathbb{Z}\Gamma $ by Theorem \ref{t:ergodic2}}.
	\end{rema}

As a corollary to the Theorems \ref{t:ergodic1} -- \ref{t:ergodic2} and Remark \ref{r:analytic zero divisor} we obtain the following results by Hayes.

	\begin{coro}[\protect{\cite[Theorem 2.3.6 and Corollary 2.5.5]{Hayes}}]
	\label{c:divisor}
Suppose that $\Gamma $ satisfies either of the following conditions.
	\begin{itemize}
\item[(1)] $\Gamma $ is an infinite, torsion-free, nilpotent group not isomorphic to the integers,
	\item[(2)] $\Gamma $ is the free group with $k\ge2$ generators.
	\end{itemize}
Then the principal $\Gamma $-action $(X_f,\alpha _f)$ is ergodic for every $f\in\mathbb{Z}\Gamma $.
	\end{coro}

	\begin{proof}
If $f=0$, then $\alpha _f$ is the left shift-action by $\Gamma $ on $X_f=\mathbb{T}^\Gamma $, which is obviously ergodic. Suppose therefore that $f\ne 0$. Since $\Gamma $ is either free or torsion-free nilpotent, $\ker \tilde{\rho }^{f^*}=\{0\}$ by Remark \ref{r:analytic zero divisor}, so that $\alpha _f$ is ergodic by either Theorem \ref{t:ergodic1} or \ref{t:ergodic2}.
	\end{proof}

Whereas the proofs of these results in \cite{Hayes} use structure theory of $\Gamma $, the proofs in this paper employ cohomological methods.

\section{Cohomological results}\label{s:cohomology}

Let $\Gamma $ be a countably infinite discrete group and $\mathscr{M}$ a left $\mathbb{Z}\Gamma $-module. A map $c\colon \Gamma \longrightarrow \mathscr{M}$ is a \textit{$1$-cocycle} (or, for our purposes here, simply a \textit{cocycle}) if
	\begin{equation}
	\label{eq:cocycle}
c(\gamma \gamma ')=c(\gamma ) + \gamma c(\gamma ')
	\end{equation}
for all $\gamma ,\gamma '\in\Gamma $. A cocycle $c\colon \Gamma \longrightarrow \mathscr{M}$ is a \textit{coboundary} (or \textit{trivial}) if there exists a $b\in \mathscr{M}$ such that
	\begin{equation}
	\label{eq:coboundary}
c(\gamma )=b-\gamma b
	\end{equation}
for every $\gamma \in \Gamma $.

A finitely generated group $G$ has two ends if and only if it is infinite and virtually cyclic, i.e., if and only if it contains a finite-index subgroup $G'\cong \mathbb{Z}$. Stallings' theorem (\cite{Stallings}) implies that a finitely generated group $G$ has a single end
whenever it is amenable and not virtually cyclic (see \cite{MV} for a short proof).

\begin{prop}
	\label{p:end}
Let $\Gamma $ be a countably infinite discrete group and $\Delta \subset \Gamma $ a finitely generated subgroup with a single end. Then every cocycle $c\colon \Delta \longrightarrow \mathbb{Z}\Gamma $ is a coboundary.
	\end{prop}
\begin{proof} By \cite[Theorem 4.6]{Houghton} if $\Delta$ has a single end, then every
$1$-cocycle $\Delta \longrightarrow \mathbb{Z}\Delta$  is a coboundary.\footnote{The authors are grateful to Andreas Thom for alerting us to this reference.}  It follows that for each $\gamma \in \Gamma$ there is some $b_\gamma \in \mathbb{Z}[\Delta\gamma]$ such that the restriction of $c(\delta)$ on $\Delta\gamma$ is equal to $b_\gamma-\delta b_\gamma$ for all $\delta\in \Delta$.

For each $\delta \in \Delta$, there is a finite set $W_\delta$ of right cosets of $\Delta$ in $\Gamma$ such that the support of $c(\delta)$ is contained in $\bigcup_{\Delta \gamma \in W_\delta}\Delta \gamma$. If $F$ is a finite symmetric set of generators of $\Delta $, then for any $\Delta \gamma \not \in \bigcup_{\delta' \in F}W_{\delta'}$, one has  $(1-\delta )\cdot b_\gamma=0$ for every $\delta\in F$ and hence for every $\delta \in \Delta $. Therefore
$c(\delta)$ is equal to $0$ on $\Delta\gamma$ for all $\Delta \gamma \not \in \bigcup_{\delta '\in F}W_{\delta '}$ and $\delta\in \Delta$. Set $b=\sum_{\Delta\gamma \in \bigcup_{\delta'\in F}W_{\delta'}}b_\gamma\in \mathbb{Z}\Gamma$. Then $c(\delta)=(1-\delta)\cdot b$ for all $\delta\in \Delta$.
    \end{proof}

Next we prove an analogous result for nonamenable groups with vanishing first $\ell^2$-Betti number, e.g., infinite property T groups \cite[Corollary 6]{BV}.

	\begin{prop}
	\label{p:T}
Let $\Gamma $ be a countably infinite discrete group and $\Delta \subset \Gamma $ a nonamenable subgroup with $\beta_1^{(2)}(\Delta) = 0$. Then every cocycle $c\colon \Delta \longrightarrow \mathbb{Z}\Gamma $ is a coboundary.
	\end{prop}

For the proof of Proposition \ref{p:T} we have to discuss cocycles of $\Gamma $ which take values in a Hilbert space $\mathscr{H}$ carrying a unitary action $U\colon \gamma \mapsto U^\gamma $ of $\Gamma $. A map $c\colon \Gamma \longrightarrow \mathscr{H}$ is a \textit{$1$-cocycle} for $U$ if
	\begin{equation}
	\label{eq:cocycle2}
c(\gamma \gamma ') = c(\gamma ) +U^\gamma c(\gamma ')
	\end{equation}
for all $\gamma ,\gamma '\in \Gamma $, and such a cocycle is a \textit{coboundary} if and only if there exists a $b\in \mathscr{H}$ with
	\begin{equation}
	\label{eq:coboundary3}
c(\gamma )=b-U^\gamma b
	\end{equation}
for every $\gamma \in \Gamma $. The cocycle $c$ is an \textit{approximate coboundary} if there exists a sequence $(c_n)_{n\ge1}$ of coboundaries $c_n\colon \Gamma \longrightarrow \mathscr{H}$ such that
	\begin{equation}
	\label{eq:approximate}
\lim_{n\to\infty }\|c_n(\gamma )-c(\gamma )\|=0
	\end{equation}
for every $\gamma \in \Gamma $.

The following lemma is well-known (cf. \cite[Proposition 1.6]{Shalom}). For convenience of the reader, we give a proof here.

	\begin{lemm}
	\label{l:coboundary2}
Let $U$ be a unitary representation of $\Gamma $ on $\mathscr{H}$ which does not contain the trivial representation weakly. Then every approximate coboundary $c\colon \Gamma \longrightarrow \mathscr{H}$ for $U$ is a coboundary.
	\end{lemm}

	\begin{proof}
Since $U$ does not weakly contain the trivial representation of $\Gamma$, we can find a finite subset $F\subset \Gamma$ and  some $\varepsilon >0$ such that
		\begin{displaymath}
\sum\nolimits_{\delta \in F}\|v-U^\delta v\|\ge \varepsilon \|v\|
	\end{displaymath}
for all $v\in \mathscr{H}$.

Let $c$ be an approximate coboundary of $\Gamma $ taking values in $\mathscr{H}$. Let $(b_n)_{n\ge1}$ be a sequence in $\mathscr{H}$ such that the coboundaries $c_n(\gamma )=b_n-U^\gamma b_n,\;\gamma \in \Gamma $, approximate $c$ in the sense of \eqref{eq:approximate}. Then
	\begin{displaymath}
\sum\nolimits_{\delta \in F}\|c(\delta )\|=\lim_{n\to \infty}\sum\nolimits_{\delta \in F}\|c_n(\delta )\|\ge \varepsilon \limsup_{n\to \infty}\|b_n\|,
	\end{displaymath}
and hence
	\begin{displaymath}
\|c(\gamma )\|=\lim_{n\to \infty}\|c_n(\gamma )\|\le 2\limsup_{n\to \infty}\|b_n\|\le 2\varepsilon^{-1}\sum\nolimits_{\delta \in F}\|c(\delta )\|
	\end{displaymath}
for all $\gamma \in \Gamma $.

For a bounded subset $Y$ of $\mathscr{H}$ and $v\in \mathscr{H}$, set $d(v, Y)=\sup_{y\in Y}\|v-y\|$.
Since $\mathscr{H}$ is a Hilbert space,  the function $v\mapsto d(v, Y)$ on $\mathscr{H}$ takes a minimal value at exactly one point, namely the Chebyshev center of $Y$, which we denote by $\rc(Y)$.

Consider the affine isometric action $V$ of $\Gamma $ on $\mathscr{H}$ defined by $V^\gamma v=U^\gamma v+c(\gamma )$ for all $\gamma \in \Gamma$ and $v\in \mathscr{H}$. Set $Y=\{c(\gamma '): \gamma '\in \Gamma\}$, and let $\gamma \in \Gamma$. Since $V^{\gamma }(Y)=Y$, we obtain that $V^{\gamma }(\rc(Y))=\rc(Y)$ and hence that $U^\gamma (\rc(Y))+c(\gamma )=\rc(Y)$. Thus $c(\gamma )=\rc(Y) - U^\gamma (\rc(Y))$ for all $\gamma \in \Gamma$, so that $c$ is a coboundary.
	\end{proof}

	\begin{proof}[Proof of Proposition~\ref{p:T}]
By \cite{BV} in the finitely generated case, and \cite[Corollary 2.4]{PT} in general, if  $\Delta$ is nonamenable and $\beta_1^{(2)}(\Delta) = 0$, then every
$1$-cocycle $\Delta \longrightarrow \ell^2(\Delta, \mathbb{R})$ for the left regular representation is a coboundary. It follows that for each $\gamma \in \Gamma$ there is some $b_\gamma \in \ell^2(\Delta\gamma, \mathbb{R})$ such that the restriction of $c(\delta)$ on $\Delta\gamma$ is equal to $b_\gamma-\delta b_\gamma$ for all $\delta\in \Delta$. Since $c(\delta)$ has finite support for each $\delta\in \Delta$, we conclude that the cocycle $c\colon \Delta \longrightarrow \ell^2(\Gamma, \mathbb{R})$ is an approximate coboundary.

Because $\Delta$ is nonamenable, its left regular representation on $\ell^2(\Delta, \mathbb{R})$ does not contain the trivial representation weakly. Since the restriction of the left regular representation of $\Gamma$ on $\ell^2(\Gamma, \mathbb{R})$ to $\Delta$ is a direct sum of copies of the left regular representation of $\Delta$, it does not contain the trivial representation of $\Delta$ weakly either. By Lemma~\ref{l:coboundary2}
there exists $v\in\ell ^2(\Gamma ,\mathbb{R})$ satisfying
\begin{equation}
	\label{eq:coboundary2}
c(\delta )=v-\tilde{\lambda }^\delta v = (1-\delta )v
	\end{equation}
for every $\delta \in\Delta $.

Since $\Delta $ is nonamenable, it is infinite. It follows that that $v\in\ell ^2(\Gamma ,\mathbb{Z})=\mathbb{Z}\Gamma $.
	\end{proof}

If a subgroup $\Delta \subset \Gamma $ has more than one end then there exist nontrivial cocycles $c\colon \Delta \longrightarrow \mathbb{Z}\Delta $ (cf. \cite[5.2. Satz IV]{Specker} or \cite[Lemma 3.5]{Swan}), which immediately implies the existence of nontrivial cocycles $c\colon \Delta \longrightarrow \mathbb{Z}\Gamma $. For example, if $\Delta $ is the free group on $k\ge2$ generators, it has nontrivial cocycles. However, Proposition \ref{p:nonamenable} below guarantees triviality of cocycles which become trivial under right multiplication by an element $f\in \mathbb{Z}\Gamma $ satisfying \eqref{eq:kernel} (cf. Remark \ref{r:analytic zero divisor}).

	\begin{prop}
	\label{p:nonamenable}
Let $\Gamma $ be a countably infinite discrete group, $\Delta \subset \Gamma $ a nonamenable subgroup, and let $f\in \mathbb{Z}\Gamma $ satisfy that $\ker \tilde{\rho }^{f^*} = \{0\}$. If $c\colon \Delta \longrightarrow \mathbb{Z}\Gamma $ is a cocycle such that $cf$ is a coboundary, then $c$ is a coboundary.
	\end{prop}

	\begin{lemm}
	\label{l:nonamenable}
Let $\Gamma $ be a countably infinite discrete group, $\Delta \subset \Gamma $ a nonamenable subgroup, and let $f\in \mathbb{Z}\Gamma $. We write $\tilde{\lambda }_\Delta $ for the unitary representation of $\Delta $ obtained by restricting the left regular representation $\tilde{\lambda }$ of $\Gamma $ on $\ell ^2(\Gamma ,\mathbb{C})$ to $\Delta $.

If $c\colon \Delta \longrightarrow \ell ^2(\Gamma ,\mathbb{C})$ is a cocycle for $\tilde{\lambda }_\Delta $ such that $c\cdot f=\tilde{\rho }^{f^*}\negthinspace \negthinspace c$ is a coboundary and $c(\Delta)$ is contained in the orthogonal complement $V$ of $\ker\tilde{\rho}^{f^*}$ in $\ell^2(\Gamma, \mathbb{C})$ \textup{(}cf. \eqref{eq:kernel}\textup{)}, then $c$ is a coboundary.
	\end{lemm}

	\begin{proof}
By assumption there exists a $b\in \ell ^2(\Gamma ,\mathbb{C})$ such that $(1-\delta )\cdot b=c(\delta )\cdot f$ for every $\delta \in \Delta $. Let $\tilde{\rho }^{f^*}=UH$ be the polar decomposition \cite[Theorem 6.1.2]{KR} of $\tilde{\rho }^{f^*}$, where $U$ is a partial isometry on $\ell ^2(\Gamma ,\mathbb{C})$, $H= \bigl(\tilde{\rho }^{ff^*}\bigr)^{1/2} = (\tilde{\rho }^f \tilde{\rho }^{f^*})^{1/2}$, and both $U$ and $H$ lie in the left-equivariant group von Neumann algebra $\mathcal{N}\Gamma $.

Note that $\ker H=\ker \tilde{\rho }^{f^*}$. We write $H=\int _{0}^{\|\tilde{\rho }^{f^*}\|}\lambda \,dE_\lambda $ for the spectral decomposition of the positive self-adjoint operator $H$ and consider, for each $0<\varepsilon<\|\tilde{\rho }^{f^*}\|$, the projection operator $P_\varepsilon = P- E_\varepsilon $, where $P$ is the orthogonal projection $\ell ^2(\Gamma ,\mathbb{C})\longrightarrow V$. Then one has $P_\varepsilon\to P$ in the strong operator topology as $\varepsilon\searrow 0$.

Put $Q_\varepsilon =UP_\varepsilon U^*$ for every $\varepsilon $ with $0<\varepsilon <\|\rho ^{f^*}\|$. Then
	\begin{align*}
P_\varepsilon (c(\delta ))\cdot f &=\tilde{\rho }^{f^*} \negthinspace P_\varepsilon (c(\delta ))=UHP_\varepsilon (c(\delta ))=UP_\varepsilon H(c(\delta )) = Q_\varepsilon UH (c(\delta ))
	\\
&= Q_\varepsilon \tilde{\rho }^{f^*}(c(\delta )) = Q_\varepsilon (c(\delta )\cdot f) =Q_\varepsilon ((1-\delta )\cdot b)=(1-\delta )\cdot Q_\varepsilon (b)
	\end{align*}
for every $\delta \in \Delta $. Since $\|\tilde{\rho }^{f^*}v\|\ge \varepsilon \|v\|$ for every $v\in {\rm range}(P_\varepsilon)$, there exists $V_\varepsilon \in \mathcal{N}\Gamma$ vanishing on the orthogonal complement of $\textup{range}(\tilde{\rho }^{f^*}\negthinspace P_\varepsilon)$ and satisfying that $V_\varepsilon \,\tilde{\rho }^{f^*}\negthinspace v = v$ for every $v\in \textup{range}(P_\varepsilon)$. Therefore
	\begin{displaymath}
P_\varepsilon (c(\delta ))=V_\varepsilon  \,\tilde{\rho }^{f^*}\negthinspace P_\varepsilon (c(\delta )) =V_\varepsilon Q_\varepsilon ((1-\delta )\cdot b))=(1-\delta )\cdot V_\varepsilon Q_\varepsilon (b).
	\end{displaymath}
The $1$-cocycle $\delta \mapsto P_\varepsilon c(\delta )=(1-\delta )\cdot V_\varepsilon Q_\varepsilon (b)$ for $\tilde{\lambda }_\Delta $ is thus a coboundary. Since $P_\varepsilon (c(\delta ))\rightarrow c(\delta )$ in $\ell^2(\Gamma ,\mathbb{C})$ as $\varepsilon\searrow 0$ for every $\delta \in \Delta $, we conclude that the $1$-cocycle $c\colon \Delta \longrightarrow \ell ^2(\Gamma ,\mathbb{C})$ for $\tilde{\lambda }_\Delta $ is an approximate coboundary.

Since $\Delta $ is nonamenable, the left regular representation of $\Delta $ on $\ell^2(\Delta ,\mathbb{C})$ does not weakly contain the trivial representation of $\Delta $. Thus, the representation $\tilde{\lambda }_\Delta $ of $\Delta $ on $\ell^2(\Gamma ,\mathbb{C})$, as a direct sum of copies of the left regular representation of $\Delta $, does not weakly contain the trivial representation of $\Delta $.

From Lemma~\ref{l:coboundary2} we conclude that there is some $b\in \ell^2(\Gamma ,\mathbb{C})$ satisfying $c(\delta )=(1-\delta)b$ for every $\delta \in \Delta $.
	\end{proof}

	\begin{proof}[Proof of Proposition \ref{p:nonamenable}]
Suppose that $f\in \mathbb{Z}\Gamma $ satisfies \eqref{eq:kernel}, and that $c\colon \Delta \longrightarrow \mathbb{Z}\Gamma $ is a $1$-cocycle such that $cf$ is a coboundary. Then $cf$ is also a coboundary when $c$ is viewed as an $\ell ^2(\Gamma ,\mathbb{C})$-valued cocycle for the unitary representation $\tilde{\lambda }_\Delta $ on $\ell ^2(\Gamma ,\mathbb{C})$. Lemma \ref{l:nonamenable} shows that there exists a $b\in \ell ^2(\Gamma ,\mathbb{C})$ such that $c(\delta )=(1-\delta )\cdot b$ for every $\delta \in \Delta $. In order to prove that $b\in \mathbb{Z}\Gamma $ we set, for every $\varepsilon >0$, $F_\varepsilon (b)=\{\gamma \in \Gamma :|b_\gamma |\ge \varepsilon \}$. Then $F_\varepsilon $ is finite, and so is the set $\{\delta \in \Delta :|(\delta \cdot b)_\gamma |=|b_{\delta ^{-1}\gamma }|\ge \varepsilon \}= \{\delta \in \Delta :\delta ^{-1}\gamma \in F_\varepsilon \}=\gamma F_\varepsilon ^{-1} \cap \Delta $ for every $\gamma \in \Gamma $. Since $\Delta $ is nonamenable, it is infinite, and by varying $\varepsilon $ we see that $\lim_{\delta \to\infty }(\delta \cdot b)_\gamma =0$ for every $\gamma \in \Gamma $. Since $c(\delta  )_\gamma =b_\gamma -(\delta \cdot b)_\gamma \in \mathbb{Z}$ we conclude, by letting $\delta \to\infty $, that $b_\gamma \in \mathbb{Z}$ for every $\gamma \in \Gamma $. This completes the proof of the proposition.
	\end{proof}

\section{Ergodicity of principal actions}\label{s:ergodicity}

We recall the following result from \cite[Lemma 1.2 and Theorem 1.6]{DSAO}.

	\begin{theo}
	\label{t:erg-mix}
If $\alpha $ is an algebraic action of a countably infinite discrete group $\Gamma $ on a compact abelian group $X$ with dual group $\hat{X}$, then $\alpha $ is ergodic if and only if the orbit $\{\hat{\alpha }^\gamma a:\gamma \in \Gamma \}$ is infinite for every nontrivial $a\in \hat{X}$.
	\end{theo}

	\begin{coro}
	\label{c:erg-mix}
Let $\Gamma $ be a countably infinite discrete group, $f\in\mathbb{Z}\Gamma $, and let $\alpha _f$ be the principal algebraic $\Gamma $-action on the group $X_f$ with Haar measure $\mu _f$ \textup{(}cf. Definition \ref{d:principal}\textup{)}. For $a\in\mathbb{Z}\Gamma /\mathbb{Z}\Gamma f=\widehat{X_f}$ let $S(a)=\{\gamma \in\Gamma : \gamma \cdot a = a\}$ be its stabilizer.

Then $\alpha _f$ is ergodic with respect to $\mu_f$ if and only if $S(a)$ has infinite index in $\Gamma $ for every nonzero $a\in\mathbb{Z}\Gamma /\mathbb{Z}\Gamma f$.
	\end{coro}

{\renewcommand\qed{}
	\begin{proof}[Proof of Theorem \ref{t:ergodic1}]
Suppose that $f\in\mathbb{Z}\Gamma $ is not a right zero-divisor, but that $\alpha _f$ is nonergodic. By Corollary \ref{c:erg-mix} there exists an $h\in\mathbb{Z}\Gamma $ such that $h\notin \mathbb{Z}\Gamma f$ and the $\Gamma $-orbit $D=\{\gamma h+\mathbb{Z}\Gamma f:\gamma \in \Gamma \}$ of $a=h+\mathbb{Z}\Gamma f$ in $\mathbb{Z}\Gamma /\mathbb{Z}\Gamma f$ is finite. We denote by
	\begin{equation}
	\label{eq:subgroup}
\Delta =\{\delta \in\Gamma :\delta h-h\in\mathbb{Z}\Gamma f\}
	\end{equation}
the stabilizer of $a$, which has finite index in $\Gamma $ by hypothesis, and consider the cocycle $c\colon \Delta \longrightarrow \mathbb{Z}\Gamma $ given by
	\begin{equation}
	\label{eq:stability-cocycle}
h-\delta h=c(\delta )f
	\end{equation}
for every $\delta \in\Delta $ (here we are using that $f$ is not a right zero-divisor). If $\Delta_0 \subset \Delta$ is an infinite subgroup on which $c$ is a coboundary then $c(\delta )=b-\delta b$ for some $b\in \mathbb{Z}\Gamma $ and every $\delta \in \Delta_0 $. Hence $c(\delta )f=(1-\delta )bf=(1-\delta )h$ for every $\delta \in\Delta_0$. Since $\Delta_0 $ is infinite, this implies that $h=bf\in \mathbb{Z}\Gamma f$, contrary to our choice of $h$. In other words, if $c$ is a coboundary when restricted to any infinite subgroup, we run into a contradiction with our assumption that $\alpha _f$ is nonergodic.
	\end{proof}

	\begin{proof}[Proof of (1)]
If $\Gamma_0 \subset \Gamma$ is a finitely generated subgroup with a single end, then the same is true for its finite-index subgroup $\Delta \cap \Gamma_0$ where $\Delta$ is from \eqref{eq:subgroup}. Proposition \ref{p:end} shows that $c$ is a coboundary on $\Delta \cap \Gamma_0$.

As was explained at the beginning of the proof of this theorem this contradicts the non-ergodicity of $\alpha _f$.
	\end{proof}

	\begin{proof}[Proof of (2)]
This is \cite[Theorem 2.4.1]{Hayes}. For convenience of the reader we include the proof. Let $\Gamma _0\subset \Gamma $ be the subgroup generated by $\textrm{supp}(h)\cup \textup{supp}(f)$. Since $\Gamma $ is not finitely generated there exists an increasing sequence of subgroups $\Gamma _n\subset \Gamma ,\,n\ge1$, such that $\Gamma _{n+1}$ is generated over $\Gamma _n$ by a single element $\gamma _{n+1}\in\Gamma _{n+1}\smallsetminus \Gamma _n$. Put $D=\{\gamma h+\mathbb{Z}\Gamma f:\gamma \in\Gamma \}\subset \mathbb{Z}\Gamma /\mathbb{Z}\Gamma f$ and $D_n=\{\gamma h+\mathbb{Z}\Gamma f:\gamma \in\Gamma _n\}\subset \mathbb{Z}\Gamma /\mathbb{Z}\Gamma f,\,n\ge0$. Then $|D_0|\le |D_1|\le \cdots \le |D_n|\le \cdots \le |D|<\infty $. Hence there exists an $N\ge0$ with $\gamma _{N+1}h+\mathbb{Z}\Gamma f=\gamma 'h +\mathbb{Z}\Gamma f$ for some $\gamma '\in \Gamma _N$. Then $(\gamma _{N+1}-\gamma ')h=gf$ for some $g\in\mathbb{Z}\Gamma $. We write $g=g_1+g_2$ with $\textup{supp}(g_1)\subset \Gamma _N$ and $\textup{supp}(g_2)\cap \Gamma _N=\varnothing $. Then
	\begin{equation}
	\label{eq:hayes3}
\gamma _{N+1}h-g_2f = g_1f + \gamma 'h.
	\end{equation}
All the terms on the right hand side of \eqref{eq:hayes3} are supported in $\Gamma _N$, whereas the supports of the terms on the left hand side of \eqref{eq:hayes3} are disjoint from $\Gamma _N$. Hence both sides of \eqref{eq:hayes3} have to vanish, which means that $\gamma _{N+1}h=g_2f$ and $h\in\gamma _{N+1}^{-1}g_2f\in \mathbb{Z}\Gamma f$, contrary to our choice of $h$. As explained above, this contradiction proves the ergodicity of $\alpha _f$.
	\end{proof}}

	\begin{proof}[Proof of (3)]
If $\Gamma_0 \subset \Gamma$ is a nonamenable subgroup with $\beta_1^{(2)}(\Gamma_0) = 0$, then the same is true for its finite-index subgroup $\Delta \cap \Gamma_0$. By Proposition \ref{p:T}, the cocycle $c\colon \Delta \cap \Gamma_0 \longrightarrow \mathbb{Z}\Gamma $ is a coboundary, which leads to a contradiction as in (1).
    \end{proof}

	\begin{proof}[Proof of Theorem \ref{t:ergodic2}]
If $\Gamma $ is amenable, use Theorem \ref{t:ergodic1} (1) or (2). If $\Gamma $ is nonamenable, combine the argument at the beginning of the proof of Theorem \ref{t:ergodic1} with Proposition \ref{p:nonamenable}.
	\end{proof}

\end{document}